\documentclass{amsart} 
\usepackage{amscd}
\usepackage{amsmath,empheq}
\usepackage{amsfonts}
\usepackage{amssymb}
\usepackage{mathrsfs} 
\usepackage[all]{xy} 
\usepackage{mathtools}
\newtheorem{theorem}{Theorem}[section]

\newtheorem{prop}[theorem]{Proposition}

\newenvironment{proof-sketch}{\noindent{\bf Sketch of Proof}\hspace*{1em}}{\qed\bigskip}

\everymath{\displaystyle}
\newcommand{\RR}{\mathbb R}

\newcommand{\NN}{\mathbb N}

\newcommand{\ZZ}{\mathbb Z}

\newcommand{\bb}{\begin{equation}}
\newcommand{\bbb}{\end{equation}}

\renewcommand{\leq}{\leqslant}

\renewcommand{\geq}{\geqslant}
\begin{document}
\title[Superlinear perturbations of the Robin eigenvalue problem]{Superlinear perturbations of the eigenvalue \\
problem for the Robin Laplacian plus an \\ indefinite and unbounded potential} 
\author[N.S. Papageorgiou]{Nikolaos S. Papageorgiou}
\address[N.S. Papageorgiou]{Department of Mathematics, National Technical University,
				Zografou Campus, Athens 15780, Greece \& Institute of Mathematics, Physics and Mechanics, 1000 Ljubljana, Slovenia}
\email{\tt npapg@math.ntua.gr}
\author[V.D. R\u{a}dulescu]{Vicen\c{t}iu D. R\u{a}dulescu}
\address[V.D. R\u{a}dulescu]{Faculty of Applied Mathematics, AGH University of Science and Technology, al. Mickiewicza 30, 30-059 Krak\'{o}w, Poland  \& Department of Mathematics, University of Craiova, 200585 Craiova, Romania \& Institute of Mathematics, Physics and Mechanics, 1000 Ljubljana, Slovenia}
\email{\tt radulescu@inf.ucv.ro}
\author[D.D. Repov\v{s}]{Du\v{s}an D. Repov\v{s}}
\address[D.D. Repov\v{s}]{ Faculty of Education and Faculty of Mathematics and Physics, University of Ljubljana,  \& Institute of Mathematics, Physics and Mechanics, 1000 Ljubljana, Slovenia}
\email{\tt dusan.repovs@guest.arnes.si}
\begin{abstract}
We consider a superlinear perturbation of the eigenvalue problem for the Robin Laplacian plus an indefinite and unbounded potential. Using variational tools and critical groups, we show that when $\lambda$ is close to a nonprincipal eigenvalue, then the problem has seven nontrivial solutions. We provide sign information for six of them.
\end{abstract}
\maketitle

\section{Introduction}

Let $\Omega\subseteq \RR^N$ ($N\geq2$) be a bounded domain with a $C^2$-boundary $\partial\Omega$. In this paper, we study the following parametric semilinear Robin problem
\begin{equation}\tag{\mbox{$P_\lambda$}}
  \left\{
\begin{array}{lll}
-\Delta u(z)+\xi(z)u(z)=\lambda u(z)+f(z,u(z)) \text{ in } \Omega,\\
\frac{\partial u}{\partial n}+\beta(z)u=0 \mbox{ on }\partial\Omega,\ \lambda\in\RR.
\end{array}
\right.
\end{equation}

In this problem, $\xi\in L^s(\Omega)$ with $s>N$ and $\xi(\cdot)$ is indefinite (that is, sign-changing). We assume that $\xi(\cdot)$ is bounded from above (that is, $\xi^+\in L^\infty(\Omega)$). So, the differential operator (the left-hand side) of problem $(P_\lambda)$ is not coercive. In the reaction (the right-hand side) of $(P_\lambda)$, we have the parametric linear term $u\mapsto \lambda u$ and a perturbation $f(z,x)$ which is a measurable function such that $f(z,\cdot)$ is continuously differentiable. We assume that $f(z,\cdot)$ exhibits superlinear growth near $\pm\infty$, but without satisfying the (usual in such cases) Ambrosetti-Rabinowitz condition (the $AR$-condition for short). Instead, we employ a less restrictive condition that incorporates in our framework superlinear nonlinearities with slower growth near $\pm \infty$ which fail to satisfy the $AR$-condition. So, problem $(P_\lambda)$ can be viewed as a perturbation of the classical eigenvalue problem for the operator $u\mapsto-\Delta u+\xi(z)u$ with Robin boundary condition.

In the past, such problems were studied primarily in the context of Dirichlet equations with no potential term. The first work was by Mugnai \cite{5Mug}, who used a general linking theorem of Marino \& Saccon \cite{4Mar-Sac} to produce three nontrivial solutions. The work of Mugnai was extended by Rabinowitz, Su \& Wang \cite{15Rab-Su-Wan} who based their method of proof on bifurcation theory, variational techniques and critical groups, in order to produce three nontrivial solutions. Analogous results for scalar periodic equations were proved by Su \& Zeng \cite{16Su-Zen}. All the aforementioned works used the $AR$-condition to express the superlinearity of the perturbation $f(z,\cdot)$. A more general superlinearity condition was employed by Ou \& Li \cite{7Ou-Li} who also produced three nontrivial solutions for $\lambda>0$ near a nonprincipal eigenvalue. As we have already mentioned earlier,  there is no potential term in all the aforementioned works, and so the differential operator is coercive. This facilitates the analysis of the problem. Papageorgiou, R\u{a}dulescu \& Repov\v{s} \cite{12Pap-Rad-Rep} went beyond Dirichlet problems and studied Robin problems with an indefinite potential. In \cite{12Pap-Rad-Rep} the emphasis was on the existence and multiplicity of positive solutions. So, the conditions on the perturbation $f(z,\cdot)$ were different, leading to a bifurcation-type result describing the change in the set of positive solutions as the parameter $\lambda$ moves in $\overset{\circ}{\RR}_+=(0,+\infty)$. We also mention the works of Castro, Cassio \& Velez \cite{1Cas-Cas-Vel}, Papageorgiou \& Papalini \cite{8Pap-Pap} (Dirichlet problems), and Hu \& Papageorgiou \cite{3Hu-Pap} (Robin problems) who also produced seven nontrivial solutions. In Castro, Cassio \& Velez \cite{1Cas-Cas-Vel} there is no potential term, while Papageorgiou \& Papalini \cite{8Pap-Pap} and Hu \& Papageorgiou \cite{3Hu-Pap} have an indefinite potential term and moreover, provide sign information for all solution they produce.
For related results we refer to Papageorgiou \& R\u{a}dulescu \cite{prans}, Papageorgiou \& Winkert \cite{pawi}, Papageorgiou \& Zhang \cite{pzana}, and Rolando \cite{rolando}.
Finally, we mention the work of Papageorgiou \& R\u{a}dulescu  \cite{11Pap-Rad} who proved multiplicity results for nearly resonant Robin problems.

In the present paper, using variational tools from the critical point theory together with suitable truncation, perturbation and comparison techniques and using also critical groups (Morse theory), we show that when the parameter $\lambda>0$ is close to an eigenvalue of $(-\Delta u+\xi u,\;H^1(\Omega))$ with Robin boundary condition, then the problem has seven nontrivial smooth solutions and we also provide sign information for six of them.

\section{Mathematical background and hypotheses}

The main spaces in the analysis of problem $(P_\lambda)$ are the Sobolev space $H^1(\Omega)$, the Banach space $C^1(\overline{\Omega})$ and the ``boundary" Lebesgue spaces $L^p(\partial\Omega)$, $1\leq p\leq \infty$.

The Sobolev space $H^1(\Omega)$ is a Hilbert space with the following inner product
$$
(u,h)=\int_\Omega uh dz+\int_\Omega (Du,Dh)_{\RR^N}dz \mbox{ for all }u,h\in H^1(\Omega).
$$

We denote by $\|\cdot\|$  the norm corresponding to this inner product. So
$$
\|u\|=\left[\|u\|_2^2+\|Du\|_2^2\right]^{1/2} \mbox{ for all }u\in H^1(\Omega).
$$

The Banach space $C^1(\overline{\Omega})$ is ordered by the positive (order) cone $$C_+=\{u\in C^1(\overline{\Omega}):\: u(z)\geq0 \mbox{ for all }z\in\overline{\Omega}\}.$$ This cone has a nonempty interior given by
$$
{\rm int}\,C_+=\{u\in C_+:\:u(z)>0 \mbox{ for all }z\in\overline{\Omega}\}.
$$

On $\partial\Omega$ we consider the $(N-1)$-dimensional Hausdorff measure (surface measure) $\tau(\cdot)$. Using this measure, we can define in the usual way the boundary value spaces $L^p(\partial\Omega)$, where $1\leq p\leq \infty$. From the theory of Sobolev spaces we know that there exists a unique continuous linear map $\gamma_0: H^1(\Omega)\to L^2(\partial\Omega)$, known as the ``trace map", such that
$$
\gamma_0(u)=u|_{\partial\Omega} \mbox{ for all }u\in H^1(\Omega)\cap C(\overline{\Omega}).
$$

So, the trace map extends the notion of boundary values to all Sobolev functions. We know that
$$
{\rm im}\, \gamma_0=H^{1/2,2}(\partial\Omega) \mbox{ and } \ker \gamma_0=H_0^1(\Omega).
$$

The linear map $\gamma_0(\cdot)$ is compact from $H^1(\Omega)$ into $L^p(\partial\Omega)$ for all $p\in\left[1,\frac{2(N-1)}{N-2}\right)$ if $N\geq3$ and into $L^p(\partial\Omega)$ for all $1\leq p<\infty$, if $N=2$.

In the sequel, for the sake of notational simplicity, we shall drop the use of the map $\gamma_0(\cdot)$. All restrictions of Sobolev functions on $\partial\Omega$ will be understood in the sense of traces.

Let $x\in\RR$. We set $x^{\pm}=\max\{\pm x,0\}$ and for any given $u\in H^1(\Omega)$ we define $u^\pm(z)=u(z)^\pm$ for all $z\in \Omega$. We know that
$$
u^\pm \in H^1(\Omega),\;u=u^+-u^-,\; |u|=u^+ +u^-.
$$

Given $u,v\in H^1(\Omega)$ with $u\leq v$, we set
$$
[u,v]=\{h\in H^1(\Omega):\:u(z)\leq h(z)\leq v(z) \mbox{ for a.a. }z\in\Omega\}.
$$

By ${\rm int}_{C^1(\overline{\Omega})}[u,v]$ we denote the interior of $[u,v]\cap C^1(\overline{\Omega})$ in the $C^1(\overline{\Omega})$-norm topology.

Let us introduce our hypotheses on the potential function $\xi(\cdot)$ and the boundary coefficient $\beta(\cdot)$.

\smallskip
$H_0:$ $\xi\in L^s(\Omega)$ with $s>N$ if $N\geq2$ and $s>1$ if $N=2$, $\xi^+\in L^\infty(\Omega)$ and $\beta\in W^{1,\infty}(\partial\Omega)$ with $\beta(z)\geq0$ for all $z\in\partial\Omega$.

\smallskip
As we have already mentioned in the introduction, our analysis of problem $(P_\lambda)$ relies on the spectrum of $u\mapsto -\Delta u+\xi(z)u$ with Robin boundary condition. So, we consider the following linear eigenvalue problem
\begin{equation}\label{eq1}
  \left\{
\begin{array}{lll}
-\Delta u(z)+\xi(z)u(z)=\hat{\lambda} u(z) \text{ in } \Omega,\\
\frac{\partial u}{\partial n}+\beta(z)u=0 \mbox{ on }\partial\Omega.
\end{array}
\right.
\end{equation}

We say that $\hat{\lambda}\in\RR$ is an ``eigenvalue", if problem \eqref{eq1} admits a nontrivial solution $\hat{u}\in H^1(\Omega)$ known as an ``eigenfunction" corresponding to the eigenvalue $\hat{\lambda}$. From hypotheses $H_0$ and the regularity theory of Wang \cite{17Wan}, we know that $\hat{u}\in C^1(\overline{\Omega})$.

Let $\gamma: H^1(\Omega)\to \RR$ be the $C^2$-functional defined by
$$
\gamma(u)=\|Du\|_2^2+\int_\Omega \xi(z) u^2dz+\int_{\partial\Omega} \beta(z) u^2 d\sigma \mbox{ for all }u\in H^1(\Omega).
$$

From D'Agui, Marano \& Papageorgiou \cite{2DAg-Mar-Pap} (see also Papageorgiou \& R\u{a}dulescu \cite{10Pap-Rad}), we know that there exists $\mu>0$ such that
\begin{equation}\label{eq2}
  \gamma(u)+\mu\|u\|_2^2\geq\hat{C}\|u\|^2 \mbox{ for some }\hat{C}>0\mbox{ and all }u\in H^1(\Omega).
\end{equation}

Using \eqref{eq2} and the spectral theorem for compact, self-adjoint operators on a Hilbert space, we show (see \cite{2DAg-Mar-Pap, 10Pap-Rad}) that the spectrum of \eqref{eq1} consists of a sequence $\{\hat{\lambda}_k\}_{k\in\NN}$ of distinct eigenvalues such that $\hat{\lambda}_k\to+\infty$ as $k\to\infty$. There is also a corresponding sequence $\{\hat{u}_k\}_{k\in\NN}\subseteq H^1(\Omega)$ of eigenfunctions which form an orthogonal basis for $H^1(\Omega)$ and an orthonormal basis for $L^2(\Omega)$. As we have already mentioned, $\hat{u}_k\in C^1(\overline{\Omega})$ for all $k\in\NN$. We denote by $E(\hat{\lambda}_k)$ the eigenspace corresponding to the eigenvalue $\hat{\lambda}_k$. We have $E(\hat{\lambda}_k)\subseteq C^1(\overline{\Omega})$ for all $k\in\NN$, this subspace is finite-dimensional and
$$
H^1(\Omega)=\overline{\underset{k\geq1}{\oplus} E(\hat{\lambda}_k)}.
$$

Moreover, each eigenspace $E(\hat{\lambda}_k)$ has the ``unique continuation property" (the UCP for short) which says that
$$\mbox{ ``if $u\in E(\hat{\lambda}_k)$ and $u(\cdot)$ vanishes on a set of positive measure,} \mbox{ then $u\equiv0$". }
$$

The first (principal) eigenvalue $\hat{\lambda}_1$ is simple, that is, $\dim E(\hat{\lambda}_1)=1$. All eigenvalues admit variational characterizations in terms of the Rayleigh quotient $\frac{\gamma(u)}{\|u\|_2^2}$, $u\in H^1(\Omega)$, $u\not=0$. We have
\begin{equation}\label{eq3}
  \hat{\lambda}=\inf\left\{\frac{\gamma(u)}{\|u\|_2^2}:\: u\in H^1(\Omega), u\not=0\right\},
\end{equation}

\begin{eqnarray}\nonumber
  \hat{\lambda}_k &=& \sup\left\{\frac{\gamma(u)}{\|u\|_2^2}:\:u\in \overline{H}_k=\underset{m=1}{\overset{k}{\oplus}}E(\hat{\lambda}_m),u\not=0\right\} \\
   &=& \inf\left\{\frac{\gamma(u)}{\|u\|_2^2}:\:u\in \hat{H}_k=\overline{\underset{m\geq k}{\oplus}E(\hat{\lambda}_m)},u\not=0\right\},\ k\geq2. \label{eq4}
\end{eqnarray}

The infimum in \eqref{eq3} is realized on $E(\hat{\lambda}_1)$, while in \eqref{eq4} both the supremum and the infimum are realized on $E(\hat{\lambda}_k)$.

It follows from \eqref{eq3} that the elements of $E(\hat{\lambda}_1)$ have fixed sign, while by \eqref{eq4} and the orthogonality of the eigenspaces, we see that the elements of $E(\hat{\lambda}_k)$ (for $k\geq2$) are nodal (that is, sign-changing). We denote by $\hat{u}_1$ the positive, $L^2$-normalized (that is, $\|\hat{u}\|_2=1$) eigenfunction corresponding to $\hat{\lambda}_1$. The regularity theory and the Hopf maximum principle imply that $\hat{u}_1\in {\rm int}\,C_+$.

Let $X$ be a Banach space, $c\in\RR$, and $\varphi\in C^1(X,\RR)$. We introduce the following sets
\begin{eqnarray*}
  K_\varphi &=& \{u\in X:\: \varphi'(u)=0\}  \mbox{ (the critical set of $\varphi$), }\\
  \varphi^c &=& \{u\in X:\: \varphi\leq c\}.
\end{eqnarray*}

We say that $\varphi(\cdot)$ satisfies the ``$C$-condition", if the following property holds:
$$
\mbox{``Every sequence $\{u_n\}_{n\geq1}$ such that }
$$
$$
\{\varphi(u_n)\}_{n\geq1}\subseteq\RR\mbox{ is bounded }
$$
$$
\mbox{ and }(1+\|u_n\|_X)\varphi'(u_n)\to0\mbox{ in }X^* \mbox{ as }n\to\infty,
$$
$$
\mbox{ admits a strongly convergent subsequence". }
$$

This is a compactness-type condition on the functional $\varphi(\cdot)$. Since the ambient space is in general not locally compact (being infinite-dimensional), the burden of compactness is passed to the functional $\varphi(\cdot)$. Using the $C$-condition one can prove a deformation theorem from which follows the minimax theory of the critical values of $\varphi(\cdot)$ (see, for example, Papageorgiou, R\u{a}dulescu \& Repov\v{s} \cite[Chapter 5]{13Pap-Rad-Rep}).

Let $(Y_1,Y_2)$ be a topological pair such that $Y_2\subseteq Y_1\subseteq X$. Given $k\in \NN_0$, we denote the $k$th-relative singular homology group for the pair $(Y_1,Y_2)$ with $\ZZ$-coefficients by  $H_k(Y_1,Y_2)$. If $\varphi\in C^1(X,\RR)$, $u\in K_\varphi$ is isolated and $c=\varphi(u)$, then the critical groups of $\varphi$ at $u$ are defined by
$$
C_k(\varphi,u)=H_k(\varphi^c\cap U, \varphi^c\cap U\setminus\{u\}) \mbox{ for all }k\in \NN_0,
$$
with $U$ being a neighborhood of $u$ such that $K_\varphi\cap\varphi^c\cap U=\{u\}$. The excision property of singular homology implies that the above definition of critical groups is independent of the choice of the isolating neighborhood $U$.

We say that a Banach space $X$  has the ``Kadec-Klee property" if the following is true
$$
``u_n \overset{w}{\to}u \mbox{ in }X \mbox{ and } \|u_n\|_X\to\|u\|_X \Rightarrow u_n\to u \mbox{ in }X".
$$

A uniformly convex space has the Kadec-Klee property. In particular, Hilbert spaces have the Kadec-Klee property.

We denote by $A\in\mathcal{L}(H^1(\Omega),H^1(\Omega)^*)$ the operator defined by
$$
\langle A(u),h\rangle=\int_\Omega (Du,Dh)_{\RR^N} dz \mbox{ for all }u,h\in H^1(\Omega).
$$

Next, we denote by $\delta_{k,m}$  the Kronecker symbol, that is,
$$
\delta_{k,m}=\left\{
               \begin{array}{ll}
                 1, & \hbox{ if }k=m \\
                 0, & \hbox{ if }k\not=m.
               \end{array}
             \right.
$$

Finally, let $2^*$ denote the Sobolev critical exponent corresponding to $2$, that is,
$$2^*=\left\{
 \begin{array}{ll}
 \frac{2N}{N-2}, & \hbox{ if }N\geq 3 \\
  +\infty, & \hbox{ if }N=2.
 \end{array}
\right.
$$

Now we introduce the hypotheses on the perturbation $f(z,x)$.

\smallskip
$H_1:$ $f:\Omega\times\RR\to\RR$ is a measurable function such that for a.a. $z\in\Omega$, $f(z,0)=0$, $f(z,\cdot)\in C^1(\RR)$ and
\begin{itemize}
  \item[$(i)$] $|f'_x(z,x)|\leq a(z)[1+|x|^{r-2}]$ for a.a. $z\in\Omega$ and all $x\in\RR$, with $a\in L^\infty(\Omega)$, $2<r<2^*$;
  \item[$(ii)$] if $F(z,x)=\displaystyle{\int_0^x f(z,s)ds}$, then $\displaystyle{\lim_{x\to\pm\infty}\frac{F(z,x)}{x^2}=+\infty}$ uniformly for a.a. $z\in\Omega$;
  \item[$(iii)$] there exists $\tau\in\left((r-2)\max\left\{1,\frac{N}{2}\right\},2^*\right)$ such that
$$
0<\hat{\beta}_0\leq\liminf_{x\to\pm\infty}\frac{f(z,x)x-2F(z,x)}{|x|^\tau} \mbox{ uniformly for a.a. }z\in\Omega;
$$
  \item[$(iv)$] $f'_x(z,0)=\displaystyle{\lim_{x\to 0}\frac{f(z,x)}{x}=0}$ uniformly for a.a. $z\in\Omega$;
  \item[$(v)$] there exist $C^*,\delta>0$ and $q>2$ such that $F(z,x)\geq-C^* |x|^q$ for a.a. $z\in\Omega$, all $x\in\RR$ and $0\leq f(z,x)x$ for a.a. $z\in\Omega$ and all $0\leq |x|\leq \delta_0$;
  \item[$(vi)$] there exist constants $C_-<0<C_+$ and $m\in\NN$, $m\geq2$ such that
$$
[\hat{\lambda}_{m+1}-\xi(z)]C_++f(z,C_+)\leq0\leq[\hat{\lambda}_{m+1}-\xi(z)]C_- +f(z,C_-) \mbox{ for a.a. }z\in\Omega;
$$
  \item[$(vii)$] for every $\rho>0$, there exists $\hat{\xi}_\rho>0$ such that for a.a. $z\in\Omega$, the function $x\mapsto f(z,x)+\hat{\xi}_\rho x$ is nondecreasing on $[-\rho,\rho]$.
\end{itemize}

\smallskip
{\em Remarks.} Hypotheses $H_1(ii),(iii)$ imply that
$$
\lim_{x\to\pm \infty}\frac{f(z,x)}{x}=\pm\infty \mbox{ uniformly for a.a. }z\in\Omega.
$$
Hence  the function $f(z,\cdot)$ is superlinear for a.a. $z\in\Omega$. However, this superlinearity of the perturbation term is not expressed using the $AR$-condition, which is common in the literature when dealing with superlinear problems. Recall that the $AR$-condition says that there exist $q>2$ and $M>0$ such that
\begin{equation}\tag{\mbox{$5a$}}
    0<q F(z,x)\leq f(z,x)x \mbox{ for a.a. }z\in\Omega \mbox{ and all }|x|\geq M
\end{equation}
\begin{equation}\tag{\mbox{$5b$}}
    0<\underset{\Omega}{{\rm essinf}}\, F(\cdot,\pm M)
\end{equation}
(see Mugnai \cite{6Mug}). Integrating $(5a)$ and using $(5b)$, we obtain the following weaker condition
\begin{eqnarray*}
  && C_0|x|^q\leq F(z,x) \mbox{ for a.a. }z\in\Omega\mbox{ and all }|x|\geq M, \\
  &\Rightarrow& C_0|x|^q\leq f(z,x)x \mbox{ for a.a. }z\in\Omega \mbox{ and all }|x|\geq M \mbox{ (see (5a)). }
\end{eqnarray*}

So we see that the $AR$-condition implies that $f(z,\cdot)$ has at least $(q-1)$-polynomial growth. In this paper, instead of the $AR$-condition, we shall employ the less restrictive condition $H_1(iii)$, which allows the consideration of superlinear nonlinearities with ``slower" growth near $\pm\infty$, which fail to satisfy the $AR$-condition. The following example illustrates this fact. For the sake of simplicity, we shall drop the $z$-dependence of $f$ and assume that $\xi\in L^\infty(\Omega)$. Suppose that for some $m\in\NN$, we have $C\geq|\hat{\lambda}_{m+2}|+\|\xi\|_\infty$, $C>0$. Then the function
$$
f(x)=\left\{
       \begin{array}{ll}
         x-(C+1)|x|^{q-2}x, & \hbox{ if }|x|\leq1\ (2<q) \\
         x\ln|x|-Cx, & \hbox{ if } 1<x
       \end{array}
     \right.
$$
satisfies hypotheses $H_1$ but fails to satisfy the $AR$-condition.

For all $\lambda>0$, let  $\varphi_\lambda: H^1(\Omega)\to \RR$  denote the energy functional associated to problem $(P_\lambda)$, which is defined by
$$
\varphi_\lambda(u)=\frac{1}{2}\gamma(u)-\frac{\lambda}{2}\|u\|_2^2-\int_\Omega F(z,u)dz \mbox{ for all }u\in H^1(\Omega).
$$

We have $\varphi_\lambda\in C^2(H^1(\Omega))$.

\section{Constant sign solutions}

In this section we shall prove the existence of four nontrivial smooth constant sign solutions when $\lambda\in(\hat{\lambda}_m,\hat{\lambda}_{m+1})$.
\begin{prop}\label{prop1}
If hypotheses $H_0$, $H_1$ hold and $\hat{\lambda}_m<\lambda<\hat{\lambda}_{m+1}$ (see $H_1(vi)$), then problem $(P_\lambda)$ has at least four nontrivial solutions of constant sign
\begin{eqnarray*}
   && u_0,\hat{u}\in{\rm int}\,C_+,\ u_0\not=\hat{u}, \\
   && v_0,\hat{v}\in-{\rm int}\,C_+,\ v_0\not=\hat{v}.
\end{eqnarray*}
\end{prop}

\begin{proof}
  Let $\mu>0$ be as in \eqref{eq2} and consider the Carath\'eodory function $g^+_\lambda(z,x)$ defined by
\begin{equation}\label{eq6}
g_\lambda^+(z,x)=\left\{
                   \begin{array}{ll}
                     (\lambda+\mu)x^+ +f(z,x^+), & \hbox{ if } x\leq C_+ \\
                     (\lambda+\mu)C_+ +f(z,C_+), & \hbox{ if } C_+<x.
                   \end{array}
                 \right.
\end{equation}

We set $G_\lambda^+=\displaystyle{\int_0^x g_\lambda^+(z,s)ds}$ and consider the $C^1$-functional $\Psi_\lambda^+:H^1(\Omega)\to\RR$ defined by
$$
\Psi_\lambda^+(u)=\frac{1}{2}\gamma(u)+\frac{\mu}{2}\|u\|_2^2-\int_\Omega G_\lambda^+(z,u)dz \mbox{ for all }u\in H^1(\Omega).
$$

From \eqref{eq2} and \eqref{eq6}, we see that $\Psi_\lambda^+(\cdot)$ is coercive. Next, using the Sobolev embedding theorem and the compactness of the trace map, we see that $\Psi_\lambda^+(\cdot)$ is sequentially weakly lower semicontinuous. So, by the Weierstrass-Tonelli theorem, we can find $u_0\in H^1(\Omega)$ such that
\begin{equation}\label{eq7}
  \Psi_\lambda^+(u_0)=\inf\left\{\Psi_\lambda^+(u):\:u\in H^1(\Omega)\right\}.
\end{equation}

Let $t>0$ be so small that $t\hat{u}_1(z)\leq \min\{C_+,\delta_0\}$ for all $z\in\overline{\Omega}$ (recall that $\hat{u}_1\in{\rm int}\,C_+$). Using \eqref{eq6} and hypothesis $H_1(v)$ we have
\begin{eqnarray*}
  && \Psi_\lambda^+ (t\hat{u}_1)\leq \frac{t^2}{2}[\hat{\lambda}_1-\lambda]<0 \mbox{ (since $\lambda>\hat{\lambda}_1$, $\|\hat{u}_1\|_2=1$), } \\
   &\Rightarrow& \Psi_\lambda^+(u_0)<0=\Psi_\lambda^+(0) \mbox{ (see \eqref{eq7}), } \\
   &\Rightarrow& u_0\not=0.
\end{eqnarray*}

From \eqref{eq7} we have
\begin{eqnarray}
&&(\Psi_\lambda^+)'(u_0)=0,\nonumber \\
  &\Rightarrow& \langle A(u_0),h\rangle+\int_\Omega [\xi(z)+\mu]u_0 hdz+\int_{\partial\Omega}\beta(z) u_0 hd\sigma=\int_\Omega g_\lambda^+(z,u_0)hdz \label{eq8} \\ \nonumber
  && \mbox{ for all }h\in H^1(\Omega).
\end{eqnarray}

In \eqref{eq8} we first  choose $h=-u_0^- \in H^1(\Omega)$. Then
\begin{eqnarray*}
   && \gamma(u_0^-)+\mu\|u_0^-\|_2^2=0 \mbox{ see \eqref{eq6}, } \\
   &\Rightarrow& \hat{C}\|u_0^-\|^2\leq0 \mbox{ (see \eqref{eq2}), } \\
   &\Rightarrow& u_0\geq0,\ u_0\not=0.
\end{eqnarray*}

Next, in \eqref{eq8} we choose $h=(u_0-C_+)^+\in H^1(\Omega)$. We have
\begin{eqnarray*}
   && \langle A(u_0), (u_0-C_+)^+\rangle+\int_\Omega[\xi(z)+\mu]u_0(u_0-C_+)^+dz+\int_{\partial\Omega} \beta(z)u_0 (u_0-C_+)^+ d\sigma \\
   &=& \int_\Omega \left[(\lambda+\mu)C_+ +f(z,C_+)\right](u_0-C_+)^+dz \mbox{ (see \eqref{eq6}) } \\
   &\leq& \int_\Omega \left[(\hat{\lambda}_{m+1}+\mu)C_+ +f(z,C_+)\right](u_0-C_+)^+dz \mbox{ (since $\lambda<\hat{\lambda}_{m+1}$) } \\
   &\leq& \int_\Omega[\xi(z)+\mu]C_+(u_0-C_+)^+dz \mbox{ (see hypotheses $H_1(vi)$), } \\
  \Rightarrow && u_0\leq C_+.
\end{eqnarray*}

So, we have proved that
\begin{equation}\label{eq9}
  u_0\in[0,C_+],\ u_0\not=0.
\end{equation}

It follows from \eqref{eq9}, \eqref{eq6} and \eqref{eq8}  that $u_0$ is a positive solution of problem $(P_\lambda)$ and we have
\begin{equation}\label{eq10}
 \left\{
\begin{array}{lll}
-\Delta u_0(z)+\xi(z)u_0(z)=\lambda u_0(z)+f(z,u_0(z)) \text{ for a.a. } z\in\Omega,\\
\frac{\partial u_0}{\partial n}+\beta(z)u_0=0 \mbox{ on }\partial\Omega
\end{array}
\right.
\end{equation}
(see Papageorgiou \& R\u{a}dulescu \cite{9Pap-Rad}).

We consider the following functions
$$
\hat{\vartheta}_\lambda(z)=\left\{
                             \begin{array}{ll}
                               0, & \hbox{ if }0\leq u_0(z)\leq 1 \\
                               \lambda-\xi(z)+\frac{f(z,u_0(z))}{u_0(z)}, & \hbox{ if } 1<u_0(z)
                             \end{array}
                           \right.
$$
and
$$
\hat{\gamma}_\lambda=\left\{
                                  \begin{array}{ll}
                                    (\lambda-\xi(z))u_0(z)+f(z,u_0(z)), & \hbox{ if }0\leq u_0(z)\leq 1 \\
                                    0, & \hbox{ if }1<u_0(z).
                                  \end{array}
                                \right.
$$

On account of hypotheses $H_0$, we have
\begin{eqnarray*}
  && \hat{\vartheta}_\lambda\in L^s(\Omega)\;(s>N) \mbox{ and }|\hat{\vartheta}_\lambda(z)|\leq|\lambda-\xi(z)|+C_1[1+u_0(z)^{r-1}] \\
   && \mbox{ for a.a. }z\in\Omega \mbox{ and some }C_1>0.
\end{eqnarray*}

If $N\geq3$ (the case $N=2$ is clear since then $2^*=+\infty$), then
$$
(r-2)\frac{N}{2}<\left[\frac{2N}{N-2}-2\right]\frac{N}{2}=\frac{2N}{N-2}=2^*.
$$

Since $u_0\in H^1(\Omega)$, by the Sobolev embedding theorem we have
\begin{eqnarray*}
   && u_0^{(r-2)N/2}\in L^1(\Omega) \\
  &\Rightarrow& \hat{\vartheta}_\lambda\in L^{\frac{N}{2}}(\Omega).
\end{eqnarray*}

From \eqref{eq10} we have
$$
\left\{
\begin{array}{lll}
-\Delta u_0(z)=\hat{\vartheta}_\lambda(z)u_0(z)+\hat{\gamma}_\lambda(z) \text{ for a.a. } z\in\Omega,\\
\frac{\partial u_0}{\partial n}+\beta(z)u_0=0 \mbox{ on }\partial\Omega.
\end{array}
\right.
$$

Using Lemma 5.1 of Wang \cite{17Wan}, we obtain that
$$
u_0\in L^\infty(\Omega).
$$
Then the Calderon-Zygmund estimates (see Lemma 5.2 of Wang \cite{17Wan}) imply that $u_0\in W^{2,s}(\Omega)$. By the Sobolev embedding theorem we have $W^{2,s}(\Omega)\hookrightarrow C^{1,\alpha}(\overline{\Omega})$ with $\alpha=1-\frac{N}{s}>0$. So, $u_0\in C^{1,\alpha}(\overline{\Omega})$.

Let $\rho=\|u\|_\infty$ and let $\hat{\xi}_\rho>0$ be as postulated by hypothesis $H_1(vii)$. From \eqref{eq10} we have
\begin{eqnarray*}
   && \Delta u_0(z)\leq (\|\xi^+\|_\infty+\hat{\xi}_\rho)u_0(z) \mbox{ for a.a. }z\in\Omega \\
   && \mbox{ (see hypotheses $H_0$), } \\
   &\Rightarrow& u_0\in {\rm int}\,C_+ \mbox{ (by the maximum principle). }
\end{eqnarray*}

Evidently, choosing $\hat{\xi}_\rho>0$ even bigger if necessary, we can deduce that for a.a. $z\in\Omega$, the function
$$
x\mapsto[\lambda+\hat{\xi}_\rho]x+f(z,x)
$$
is nondecreasing on $[-\rho,\rho]$ ($\rho=\|u_0\|_\infty$). We have
\begin{eqnarray}\nonumber
  && -\Delta u_0(z)+[\xi(z)+\hat{\xi}_\rho]u_0(z) \\ \nonumber
  && =[\lambda+\hat{\xi}_\rho]u_0(z)+f(z,u_0(z)) \\ \nonumber
  && \leq[\lambda+\hat{\xi}_\rho]C_+ +f(z,C_+) \mbox{ (see \eqref{eq9}) } \\ \nonumber
  && \leq[\xi(z)+\hat{\xi}_\rho]C_+ \mbox{ for a.a. }z\in\Omega \mbox{ (see hypothesis $H_1(vi)$), } \\ \nonumber
  &\Rightarrow& \Delta(C_+ -u_0)(z)\leq \left[\|\xi^+\|_\infty+\hat{\xi}_\rho\right](C_+-u_0(z)) \mbox{ for a.a. }z\in\Omega, \\ \nonumber
  &\Rightarrow& C_+-u_0\in{\rm int}\,C_+, \\
  &\Rightarrow& u_0\in{\rm int}_{C^1(\overline{\Omega})}[0, C_+]. \label{eq11}
\end{eqnarray}

Let $\varphi^+_\lambda: H^1(\Omega)\to \RR$ be the $C^1$-functional defined by
$$
\varphi_\lambda^+=\frac{1}{2}\gamma(u)+\frac{\mu}{2}\|u^-\|_2^2-\frac{\lambda}{2}\|u^+\|_2^2-\int_\Omega F(z,u^+)dz \mbox{ for all }u\in  H^1(\Omega).
$$

From \eqref{eq6} it is clear that
\begin{eqnarray*}
  && \varphi_\lambda^+|_{[0,C_+]}=\Psi_\lambda^+|_{[0,C_+]} \\
  &\Rightarrow& u_0 \mbox{ is a local $C^1(\overline{\Omega})$-minimizer of $\varphi_\lambda^+$ (see \eqref{eq11}), } \\
  &\Rightarrow& u_0 \mbox{ is a local $ H^1(\Omega)$-minimizer of $\varphi_\lambda^+$ } \\
  && \mbox{ (see Papageorgiou \& R\u{a}dulescu \cite{9Pap-Rad}).}
\end{eqnarray*}

It is easy to see that
\begin{eqnarray*}
  && K_{\varphi_\lambda^+}\subseteq C_+ \mbox{ (regularity theory), } \\
  &\Rightarrow& K_{\varphi_\lambda^+}\subseteq{\rm int}\,C_+ \cup\{0\} \mbox{ (maximum principle). }
\end{eqnarray*}

So, we may assume that $K_{\varphi_\lambda^+}$ is finite. Otherwise we already have an infinity of positive smooth solutions and we are done. Then on account of Theorem 5.7.6 of Papageorgiou, R\u{a}dulescu \& Repov\v{s} \cite[p. 449]{13Pap-Rad-Rep}, we can find $\rho_0\in(0,1)$ so small  that
\begin{equation}\label{eq12}
  \varphi_\lambda^+(u_0)<\inf\left\{\varphi_\lambda^+(u):\:\|u-u_0\|=\rho_0\right\}=m_\lambda^+.
\end{equation}

Hypothesis $H_1(ii)$ implies that
\begin{equation}\label{eq13}
  \varphi_\lambda^+(t\hat{u}_1)\to-\infty \mbox{ as }t\to+\infty.
\end{equation}

\smallskip {\it
Claim}. The functional $\varphi_\lambda^+$ satisfies the $C$-condition.

\smallskip
Consider a sequence $\{u_n\}_{n\geq1}\subseteq  H^1(\Omega)$ such that
\begin{eqnarray}
  && |\varphi_\lambda^+(u_n)|\leq C_2 \mbox{ for some }C_2>0\mbox{ and all }n\in\NN, \label{eq14} \\
  && (1+\|u_n\|)(\varphi_\lambda^+)'(u_n)\to0 \mbox{ in }  H^1(\Omega)^* \mbox{ as }n\to\infty. \label{eq15}
\end{eqnarray}

From \eqref{eq15} we have
\begin{eqnarray}\nonumber
 \Big|\langle A(u_n),h\rangle+\int_\Omega \xi(z)u_n hdz+\int_{\partial\Omega} \beta(z)u_n hd\sigma\!\!\!\! &-&\!\!\!\!\int_\Omega \mu u_n^- hdz-\int_\Omega[\lambda u_n^*+f(z,u_n^+)]hdz \Big| \\
   &\leq& \frac{\varepsilon_n\|h\|}{1+\|u_n\|} \label{eq16} \\ \nonumber
   && \mbox{ for all }h\in  H^1(\Omega),\mbox{ with }\varepsilon_n\to0^+.
\end{eqnarray}

In \eqref{eq16} we choose $h=-u_n^-\in  H^1(\Omega)$. Then
\begin{eqnarray}\nonumber
  && \gamma(u_n^-)+\mu\|u_n^-\|_2^2\leq \varepsilon_n \mbox{ for all }n\in\NN, \\ \nonumber
  &\Rightarrow& \hat{C}\|u_n^-\|^2\leq \varepsilon_n \mbox{ for all }n\in\NN \mbox{ (see \eqref{eq2}), } \\
  &\Rightarrow& u_n^-\to0 \mbox{ in } H^1(\Omega) \mbox{ as }n\to\infty. \label{eq17}
\end{eqnarray}

Next, we choose $h=u_n^+ \in  H^1(\Omega)$ in \eqref{eq16}. We obtain
\begin{equation}\label{eq18}
  -\gamma(u_n^+)+\int_\Omega[\lambda(u_n^+)^2+f(z,u_n^+)u_n^+]dz\leq \varepsilon_n \mbox{ for all }n\in\NN.
\end{equation}

On the other hand, from \eqref{eq14} and \eqref{eq17}, we have
\begin{equation}\label{eq19}
  \gamma(u_n^+)-\int_\Omega[\lambda(u_n^+)^2+2F(z,u_n^+)]dz\leq C_3 \mbox{ for some }C_3>0 \mbox{ and all }n\in\NN.
\end{equation}

We add \eqref{eq18} and \eqref{eq19} and obtain
\begin{equation}\label{eq20}
  \int_\Omega [f(z,u_n^+)u_n^+-2F(z,u_n^+)]dz\leq C_4 \mbox{ for some } C_4>0 \mbox{ and all }n\in\NN.
\end{equation}

Hypotheses $H_1(i),(iii)$ imply that we can find $\hat{\beta}_1\in(0,\hat{\beta}_0)$ and $C_5>0$ such that
\begin{equation}\label{eq21}
  \hat{\beta}_1|x|^\tau-C_5\leq f(z,x)x-2F(z,x) \mbox{ for a.a. }z\in\Omega\mbox{ and all }x\in\RR.
\end{equation}

We use \eqref{eq21} in \eqref{eq20} and conclude that
\begin{equation}\label{eq22}
  \{u_n^+\}_{n\geq1}\subseteq L^\tau(\Omega) \mbox{ is bounded. }
\end{equation}

First, assume that $N\geq3$. From hypothesis $H_1(iii)$ we see that without any loss of generality, we may assume that $\tau<r<2^*$. So, we can find $t\in(0,1)$ such that
\begin{equation}\label{eq23}
  \frac{1}{r}=\frac{1-t}{\tau}+\frac{t}{2^*}.
\end{equation}

From the interpolation inequality (see Proposition 2.3.17 of Papageorgiou \& Winkert \cite[p. 116]{14Pap-Win}), we have
\begin{eqnarray}\nonumber
  && \|u_n^+\|_r\leq \|u_n^+\|_\tau^{1-t}\|u_n^+\|_{2^*}^t \\
  &\Rightarrow& \|u_n^+\|_r^r\leq C_6\|u_n^+\|^{tr} \mbox{ for some } C_6>0 \mbox{ and all }n\in\NN \label{eq24} \\ \nonumber
  && \mbox{ (see \eqref{eq22} and use the Sobolev embedding theorem). }
\end{eqnarray}

From hypothesis $H_1(i)$ we have
\begin{equation}\label{eq25}
  f(z,x)x\leq C_7[1+x^r] \mbox{ for a.a. }z\in\Omega, \mbox{ all }x\geq0\mbox{ and some }C_7>0.
\end{equation}

In \eqref{eq16} we choose $h=u_n^+\in H^1(\Omega)$. Then
\begin{eqnarray}\nonumber
  \gamma(u_n^+)+\mu\|u_n^+\|_2^2 &\leq& [\lambda+\mu]\|u_n^+\|_2^2+\int_\Omega f(z,u_n^+)u_n^+dz+\varepsilon_n \\ \nonumber
   &\leq& \left[|\lambda|+\mu\right]\|u_n^+\|_2^2+C_8[1+\|u_n^+\|^{tr}] \\ \nonumber
   && \mbox{ for some }C_8>0 \mbox{ (see \eqref{eq25} and \eqref{eq24}) } \\ \nonumber
   &\leq& C_9\left[1+\|u_n^+\|^{tr}\right] \\ \nonumber
   && \mbox{ for some }C_9>0 \mbox{ (recall that $2\leq \tau$ and see \eqref{eq22}), } \\
  \Rightarrow\hat{C}\|u_n^+\|^2 &\leq& C_9\left[1+\|u_n^+\|^{tr}\right] \mbox{ for all }n\in\NN. \label{eq26}
\end{eqnarray}

Using \eqref{eq23} and the fact that $\tau>(r-2)\frac{N}{2}$ (see hypothesis $H_1(iii)$ and recall that $N\geq3$), we see that $tr<2$. So, it follows from \eqref{eq26}  that
\begin{eqnarray}\nonumber
  && \{u_n^+\}_{n\geq1}\subseteq H^1(\Omega) \mbox{ is bounded, } \\
  &\Rightarrow& \{u_n\}_{n\geq1}\subseteq H^1(\Omega) \mbox{ is bounded (see \eqref{eq17}). } \label{eq27}
\end{eqnarray}

We may assume that
\begin{equation}\label{eq28}
  u_n\overset{w}{\to}u \mbox{ in }H^1(\Omega) \mbox{ as }n\to\infty.
\end{equation}

In \eqref{eq16} we choose $h=u_n-u \in H^1(\Omega)$, pass to the limit as $n\to\infty$ and use \eqref{eq24}, the Sobolev embedding theorem and the compactness of the trace map. We obtain
\begin{eqnarray}\nonumber
  && \lim_{n\to\infty}\langle A(u_n),u_n-u\rangle=0, \\
  &\Rightarrow& \|Du_n\|_2\to\|Du\|_2. \label{eq29}
\end{eqnarray}

Invoking \eqref{eq28}, \eqref{eq29} and the Kadec-Klee property of $H^1(\Omega)$, we infer that
\begin{equation}\label{eq30}
  u_n\to u \mbox{ in }H^1(\Omega) \mbox{ as }n\to\infty.
\end{equation}

This proves that $\varphi_\lambda^+$ satisfies the $C$-condition when $N\geq3$.

If $N=2$, then $2^*=+\infty$ and by the Sobolev embedding theorem, we have $H^1(\Omega)\hookrightarrow L^\eta(\Omega)$ compactly for all $1\leq\eta<\infty$. Then the previous argument works if we replace $2^*(=+\infty)$ with $\eta>r>\tau$. We choose $t\in(0,1)$ such that
\begin{eqnarray*}
  && \frac{1}{r}=\frac{1-t}{\tau}+\frac{t}{\eta} \\
  &\Rightarrow& tr=\frac{\eta(r-t)}{\eta-\tau}, \\
  &\Rightarrow& tr\to r-\tau \mbox{ as }\eta\to+\infty \mbox{ and } r-\tau<2 \mbox{ (see $H_1(iii)$). }
\end{eqnarray*}

So, we choose $\eta>r$ big enough so that $tr<2$ and reasoning as above, we obtain \eqref{eq27} and invoking the Kadec-Klee property, we again reach \eqref{eq30}. We conclude that $\varphi_\lambda^+$ satisfies the $C$-condition. This proves the claim.

Then \eqref{eq12}, \eqref{eq13} and the claim, permit the use of the mountain pass theorem. So, we can find $\hat{u}\in H^1(\Omega)$ such that
\begin{equation}\label{eq31}
  \hat{u}\in K_{\varphi_\lambda^+}\subseteq{\rm int}\,C_+\cup\{0\} \mbox{ and } m_\lambda^+\leq \varphi_\lambda^+(\hat{u}) \mbox{ (see\eqref{eq12}). }
\end{equation}

It follows from \eqref{eq12} and \eqref{eq31}  that $\hat{u}\not=u_0$. If we show that $\hat{u}\not=0$, then this will be the second positive solution of $(P_\lambda)$.

On account of hypotheses $H_1(i),(iv)$, we have
\begin{equation}\label{eq32}
  |f(z,x)|\leq C_{10}[|x|+|x|^{r-1}] \mbox{ for a.a. }z\in\Omega, \mbox{ all }x\in\RR\mbox{ and some } C_{10}>0.
\end{equation}

We have
\begin{eqnarray}\nonumber
  |\varphi_\lambda(u)-\varphi_\lambda^+(u)| &\leq& \frac{\mu+|\lambda|}{2}\|u_n^-\|_2^2+\int_\Omega |F(z,-u^-)|dz \\
   &\leq& C_{11}\left[\|u\|^2+\|u\|^r\right] \mbox{ for some }C_{11}>0 \mbox{ (see \eqref{eq32}). }\label{eq33}
\end{eqnarray}

Also, for $h\in H^1(\Omega)$ we have
\begin{eqnarray}\nonumber
  && \left|\langle \varphi'_\lambda(u)-(\varphi_\lambda^+)'(u),h\rangle\right|\leq C_{12}\left[\|u\|+\|u\|^{r-1}\right]\|h\| \mbox{ for some }C_{12}>0, \\
  &\Rightarrow& \|\varphi'_\lambda(u)-(\varphi_\lambda^+)'(u)\|_{H^1(\Omega)^*}\leq C_{12}\left[\|u\|+\|u\|^{r-1}\right]. \label{eq34}
\end{eqnarray}

From \eqref{eq33}, \eqref{eq34} and the $C^1$-continuity of critical groups (see Theorem 6.3.4 of Papageorgiou, R\u{a}dulescu \& Repov\v{s} \cite[p. 503]{13Pap-Rad-Rep}), we have
\begin{equation}\label{eq35}
  C_k(\varphi_\lambda,0)=C_k(\varphi_\lambda^+,0) \mbox{ for all }k\in\NN_0.
\end{equation}

By hypothesis, $\lambda\in(\hat{\lambda}_m,\hat{\lambda}_{m+1})$ and $m\geq2$. So, $u=0$ is a nondegenerate critical point of $\varphi_\lambda$ with Morse index $d_m=\dim \overline{H}_m\geq2$ (since $m\geq2$). Hence by Proposition 6.2.6 of \cite[p. 479]{13Pap-Rad-Rep}, we have
\begin{eqnarray}\nonumber
  && C_k(\varphi_\lambda,0)=\delta_{k,d_m}\ZZ\mbox{ for all }k\in\NN_0, \\
  &\Rightarrow& C_k(\varphi_\lambda^+,0)=\delta_{k,d_m}\ZZ \mbox{ for all }k\in\NN_0 \mbox{ (see \eqref{eq35}). } \label{eq36}
\end{eqnarray}

On the other hand, from the previous part of the proof we know that $\hat{u}\in K_{\varphi_\lambda^+}$ is of mountain pass type. Therefore Theorem 6.5.8 of Papageorgiou, R\u{a}dulescu \& Repov\v{s} \cite[p. 527]{13Pap-Rad-Rep} implies that
\begin{equation}\label{eq37}
  C_1(\varphi_\lambda^+,\hat{u})\not=0.
\end{equation}

By \eqref{eq37}, \eqref{eq36} and since $d_m\geq2$, we can conclude that $\hat{u}\not=0$ and so $\hat{u}\in{\rm int}\,C_+$ is the second positive solution of $(P_\lambda)$ distinct from $u_0$.

For the negative solutions, we consider the Carath\'eodory function $g_\lambda^-(z,x)$ defined by
$$
g_\lambda^-(z,x)=\left\{
                   \begin{array}{ll}
                     (\lambda+\mu)C_-+f(z,C_-), & \hbox{ if }x\leq C_- \\
                     (\lambda+\mu)(-x^-)+f(z,-x^-), & \hbox{ if }C_-<x.
                   \end{array}
                 \right.
$$

We set $G_\lambda^-(z,x)=\int_0^x g_\lambda^-(z,s)ds$ and consider the $C^1$-functionals $\Psi^-_\lambda,\varphi_\lambda^-: H^1(\Omega)\to\RR$ defined by
\begin{eqnarray*}
  \Psi_\lambda^-(u) &=& \frac{1}{2}\gamma(u)+\frac{\mu}{2}\|u\|_2^2-\int_\Omega G_\lambda^-(z,u)dz \\
  \mbox{and }\varphi_\lambda^-(u) &=& \frac{1}{2}\gamma(u)+\frac{\mu}{2}\|u^+\|_2^2-\frac{\lambda}{2}\|u^-\|_2^2-\int_\Omega F(z,-u^-)dz
\end{eqnarray*}
for all $u\in H^1(\Omega)$.

Working with these two functionals as above, we produce two negative solutions $v_0,\hat{v}\in-{\rm int}\,C_+$, $v_0\not=\hat{v}$.
\end{proof}

\section{Nodal solutions}

In this section we show that when $\lambda$ is close to $\hat{\lambda}_{m+1}$ (near resonance) we can generate two nodal (sign-changing) solutions.

\begin{prop}\label{prop2}
If hypotheses $H_0$, $H_1$ hold and $\lambda\in(\hat{\lambda}_m,\hat{\lambda}_{m+1})$ (see $H_1(vi)$), then we can find $\hat{\delta}>0$ such that for all $\lambda\in(\hat{\lambda}_{m+1}-\hat{\delta},\hat{\lambda}_{m+1})$ problem $(P_\lambda)$ has at least two nodal solutions $y_0,\hat{y}\in C^1(\overline{\Omega})$.
\end{prop}

\begin{proof}
  By Proposition 2.3 of Rabinowitz, Su \& Wang \cite{15Rab-Su-Wan}, we know that there exists $\delta_1>0$ such that for all $\lambda\in(\hat{\lambda}_{m+1}-\delta_1,\hat{\lambda}_{m+1})$ problem $(P_\lambda)$ has at least two nontrivial solutions $y_0,\hat{y}\in H^1(\Omega)$. As before, using the regularity theory of Wang \cite{17Wan}, we conclude that $y_0,\hat{y}\in C^1(\overline{\Omega})$. Note that the result of Rabinowitz, Su \& Wang \cite{15Rab-Su-Wan} is for Dirichlet problems with $\xi\equiv0$. However, their proof is based on the abstract bifurcation theorem of Rabinowitz (see Theorem 2.1 in \cite{15Rab-Su-Wan}) and so it applies verbatim in our case, too.

We will show that we can have these two solutions $y_0,\hat{y}\in C^1(\overline{\Omega})$  be nodal. From the proof of Proposition 2.3 of Rabinowitz, Su \& Wang \cite{15Rab-Su-Wan} and using hypothesis $H_1(iv)$, we see that given $\varepsilon\in\left(0,\frac{\lambda-\hat{\lambda}_1}{2}\right)$ (recall that $\lambda>\hat{\lambda}_1$), we can find $0<\hat{\delta}\leq \delta_1$ such that
\begin{equation}\label{eq38}
  \lambda\in(\hat{\lambda}_{m+1}-\hat{\delta},\hat{\lambda}_{m+1})\Rightarrow|f(z,w(z))|\leq \varepsilon w(z) \mbox{ for a.a. }z\in\Omega,
\end{equation}
with $w=y_0$ or $w=\hat{y}$. Suppose that $w\in{\rm int}\,C_+$ (the reasoning is similar if $w\in-{\rm int}\,C_+$). We have
\begin{eqnarray*}
   && \hat{\lambda}_1\int_\Omega w\hat{u}_1dz \\
   &=& \langle A(\hat{u}_1), w\rangle+\int_\Omega\xi(z)\hat{u}_1 wdz+\int_{\partial\Omega}\beta(z)\hat{u}_1w d\sigma \\
   &=& \int_\Omega(-\Delta w)\hat{u}_1dz+\int_{\partial\Omega}\frac{\partial w}{\partial n}\hat{u}_1d\sigma+\int_\Omega \xi(z)\hat{u}_1 wdz+\int_{\partial\Omega}\beta(z)\hat{u}_1wd\sigma \\
   && \mbox{ (using Green's identity) } \\
   &=& \int_\Omega[\lambda w-f(z,w)]\hat{u}_1dz \mbox{ (since $w$ is a solution of $(P_\lambda)$) } \\
   &\geq& \int_\Omega \left[\lambda w-\frac{\lambda-\hat{\lambda}_1}{2}w\right]\hat{u}_1 dz \mbox{ (see \eqref{eq38} and recall that $0<\varepsilon\leq \frac{\lambda-\hat{\lambda}_1}{2}$) } \\
   &=& \int_\Omega\frac{\lambda+\hat{\lambda}_1}{2}w\hat{u}_1dz \\
   &>& \hat{\lambda}_1\int_\Omega w\hat{u}_1 dz, \mbox{ a contradiction. }
\end{eqnarray*}

So, $w=y_0$ or $w=\hat{y}$ cannot be constant sign and so $y_0,\hat{y}\in C^1(\overline{\Omega})$ are nodal solutions of $(P_\lambda)$ for $\lambda\in(\hat{\lambda}_{m+1}-\hat{\delta},\hat{\lambda}_{m+1})$.
\end{proof}

\section{The seventh nontrivial solution}

In this section we prove the existence of a seventh nontrivial solution for problem $(P_\lambda)$ when $\lambda\in(\hat{\lambda}_m,\hat{\lambda}_{m+1})$. However, we are unable to provide sign information for this seventh solution.

\begin{prop}\label{prop3}
  If hypotheses $H_0$, $H_1(i),(iv)$ hold and $\lambda<\hat{\lambda}_{m+2}$, then there exists $\rho>0$ such that
\begin{eqnarray*}
  &&\varphi_\lambda|_{\hat{H}_{m+2}\cap\partial B_\rho}\geq \tilde{C}_0>0  \\
  &\mbox{ with }& \hat{H}_{m+2}=\overline{\underset{k\geq m+2}{\oplus}E(\hat{\lambda}_k)}, B_\rho=\{u\in H^1(\Omega):\:\|u\|<\rho\}.
\end{eqnarray*}
\end{prop}

\begin{proof}
  Hypotheses $H_1(i),(iv)$ imply that given $\varepsilon>0$, we can find $C_\varepsilon>0$ such that
\begin{equation}\label{eq39}
  |F(z,x)|\leq\frac{\varepsilon}{2}x^2+C_\varepsilon|x|^r \mbox{ for a.a. }z\in\Omega \mbox{ and all }x\in\RR.
\end{equation}

Let $u\in \hat{H}_{m+2}$. We have
\begin{eqnarray*}
  \varphi_\lambda(u) &\geq& \frac{1}{2}\gamma(u)-\frac{\lambda}{2}\|u\|_2^2-\frac{\varepsilon}{2}\|u\|^2-\hat{C}_\varepsilon\|u\|^r \\
  && \mbox{ for some }\hat{C}_\varepsilon>0 \mbox{ (see \eqref{eq39}) }\\
  &\geq& \frac{C_{13}-\varepsilon}{2}\|u\|^2-\hat{C}_\varepsilon\|u\|^r \mbox{ for some }C_{13}>0 \mbox{ (recall that $\lambda<\hat{\lambda}_{m+2}$). }
\end{eqnarray*}

Choose $\varepsilon\in(0,C_{13})$. Then we obtain
$$
\varphi_\lambda(u)\geq C_{14}\|u\|^2-\hat{C}_\varepsilon\|u\|^r \mbox{ for some }C_{14}>0 \mbox{ and all }u\in \hat{H}_{m+2}.
$$

Since $2<r$, we can find $\rho\in(0,1)$ so small  that
$$
\varphi_\lambda(u)\geq\tilde{C}_0>0 \mbox{ for all } u\in \hat{H}_{m+2}\cap\partial B_\rho.
$$
The proof is now complete.
\end{proof}

Let $\hat{u}_{m+2}\in E(\hat{\lambda}_{m+2})$ with $\|\hat{u}_{m+2}\|=1$ and let $V=\overline{H}_{m+1}\oplus\RR\hat{u}_{m+2}$, with $\overline{H}_{m+1}=\underset{k=1}{\overset{m+1}{\oplus}}E(\hat{\lambda}_k)$. For $\rho_1>0$, we introduce the set
$$
C=\{u=\overline{u}+\vartheta\hat{u}_{m+2}:\:\overline{u}\in \overline{H}_{m+1},\vartheta\geq0,\|u\|\leq \rho_1\}.
$$

Evidently, we have
$$
\partial C=C_0=\left\{u=\overline{u}+\vartheta\hat{u}_{m+2}:\ \left(\overline{u}\in\overline{H}_{m+1},\vartheta\geq0,
\|u\|=\rho_1\right)\mbox{ or }\left(\overline{u}\in \overline{H}_{m+1},\|\overline{u}\|\leq\rho_1,\vartheta=0\right)\right\}.
$$

\begin{prop}\label{prop4}
  If hypotheses $H_0$, $H_1(i),(ii),(iv),(v)$ hold and $\lambda\in(\hat{\lambda}_{m},\hat{\lambda}_{m+1})$, then there exist $\rho_1>0$ and $\tilde{\delta}>0$ such that
$$
\varphi_\lambda|_{C_0}\leq \tilde{C}_1<\tilde{C}_0
$$
with $\tilde{C}_0>0$ as in Proposition \ref{prop3}.
\end{prop}

\begin{proof}
  From hypotheses $H_1(i),(ii),(v)$ given $\eta>0$, we can find $\hat{C}_\eta^*>0$ such that
\begin{equation}\label{eq40}
  F(z,x)\geq\frac{\eta}{2}x^2-\hat{C}_\eta^*|x|^q \mbox{ for a.a. }z\in\Omega \mbox{ and all }x\in\RR.
\end{equation}

The space $V$ is finite dimensional and so all norms are equivalent. Let $u\in V$. We have
\begin{eqnarray*}
  \varphi_\lambda(u) &\leq& \frac{1}{2}\gamma(u)-\frac{\lambda}{2}\|u\|_2^2-\frac{\eta}{2}\|u\|_2^2+\frac{\hat{C}_\eta^*}{q}\|u\|_q^q \mbox{ (see \eqref{eq40}) } \\
   &\leq& C_{15}\left[\hat{\lambda}_{m+2}-\lambda-\eta\right]\|u\|^2+C_{16}\|u\|^q \mbox{ for some }C_{15},C_{16}=C_{16}(\eta)>0.
\end{eqnarray*}

Since $\eta>0$ arbitrary, choosing $\eta>0$ big enough, we have
$$
\varphi_\lambda(u)\leq C_{16}\|u\|^q-C_{17}\|u\|^2 \mbox{ for some }C_{17}>0.
$$

Recall that $q>2$. Then we can find $\rho_1\in(0,1)$ so small that
$$
\varphi_\lambda|_{V\cap\partial B_\rho}\leq0<\tilde{C}_0 \mbox{ (see Proposition \ref{prop3}). }
$$

If $\overline{u}\in \overline{H}_{m+1}$, $\|\overline{u}\|\leq \rho_1$, then
\begin{eqnarray*}
  \varphi_\lambda(\overline{u}) &\leq& \frac{1}{2}\gamma(\overline{u})-\frac{\lambda}{2}\|\overline{u}\|_2^2+C^*\|\overline{u}\|_q^q \\
   &\leq& \frac{1}{2}\left[\hat{\lambda}_{m+1}-\lambda\right]\|\overline{u}\|_2^2+C^*\|\overline{u}\|_q^q \mbox{ (see $H_1(v)$) } \\
   &\leq & C_{18}\rho_1^2 \mbox{ (since $q>2$,\ $\lambda<\hat{\lambda}_{m+1}$ and $\rho_1\in(0,1)$).}
\end{eqnarray*}

Choosing $\rho_1\in(0,1)$ even smaller if necessary, we have
$$
\varphi_\lambda|_{\overline{H}_{m+1}\cap\partial B_{\rho_1}}\leq \tilde{C}_1<\tilde{C}_0
$$
for all $\lambda\in(\hat{\lambda}_m,\hat{\lambda}_{m+1})$ and with $\tilde{C}_0>0$ (as in Proposition \ref{prop3}).

Therefore we can conclude that
$$
\varphi_\lambda|_{C_0}\leq \tilde{C}_1<\tilde{C}_0 \mbox{ for }\lambda\in(\hat{\lambda}_m,\hat{\lambda}_{m+1}).
$$
The proof is now complete.
\end{proof}

Now we are ready to produce the seventh nontrivial smooth solution of problem $(P_\lambda)$.

\begin{prop}\label{prop5}
  If hypotheses $H_0$, $H_1$ hold and $\lambda\in(\hat{\lambda}_{m+1}-\hat{\delta},\hat{\lambda}_{m+1})$ (see Proposition \ref{prop2}), then problem $(P_\lambda)$ has a seventh nontrivial solution $\tilde{y}\in C^1(\overline{\Omega})$.
\end{prop}

\begin{proof}
  Let $D=\overline{H}_{m+1}\cap\partial B_{\rho_1}$. From Proposition 6.6.5 of Papageorgiou, R\u{a}dulescu \& Repov\v{s} \cite[p. 532]{13Pap-Rad-Rep}, we know that
$$
\{C,C_0\} \mbox{ and }D \mbox{ homologically link in dimension }d_{m+1}+1
$$
with $d_{m+1}=\dim\overline{H}_{m+1}$. Then Propositions \ref{prop3} and \ref{prop4} and Corollary 6.6.8 of \cite{13Pap-Rad-Rep}, imply that there exists $\tilde{y}\in K_{\varphi_\lambda}\subseteq C^1(\overline{\Omega})$ (see Wang \cite{17Wan}) such that
\begin{equation}\label{eq41}
  C_{d_{m+1}+1}(\varphi_\lambda,\tilde{y})\not=0.
\end{equation}

From the proof of Proposition \ref{prop1}, we know that $u_0\in{\rm int}\,C_+$ and $v_0\in-{\rm int}\,C_+$ are local minimizers of $\varphi_\lambda^+$ and of $\varphi_\lambda^-$, respectively. Note that
\begin{equation}\label{eq42}
  \varphi_\lambda|_{C_+}=\varphi_\lambda^+|_{C_+} \mbox{ and }\varphi_\lambda|_{-C_+}=\varphi_\lambda^+|_{-C_+}.
\end{equation}

So, it follows that $u_0\in{\rm int}\,C_+$ and $v_0\in-{\rm int}\,C_+$ are also local minimizers of $\varphi_\lambda$ (see \cite{9Pap-Rad}). Therefore we have
\begin{equation}\label{eq43}
  C_k(\varphi_\lambda,u_0)=C_k(\varphi_\lambda,v_0)=\delta_{k,0}\ZZ \mbox{ for all }k\in\NN_0.
\end{equation}

Also, again from the proof of Proposition \ref{prop1}, we know that the solutions $\hat{u}\in{\rm int}\,C_+$ and $\hat{v}\in-{\rm int}\,C_+$ are critical points of mountain pass type of the functionals $\varphi_\lambda^+$ and $\varphi_\lambda^-$ respectively. Therefore we have
\begin{equation}\label{eq44}
  C_1(\varphi_\lambda^+,\hat{u})\not=0 \mbox{ and }C_1(\varphi_\lambda^-,\hat{v})\not=0 \  \mbox{(see \eqref{eq37}).}
\end{equation}

From \eqref{eq42} and since $\hat{u}\in{\rm int}\,C_+$, $\hat{v}\in-{\rm int}\,C_+$, we have
\begin{equation}\label{eq45}
  C_k\left(\varphi_\lambda^+|_{C^1(\overline{\Omega})},\hat{u}\right)=
C_k\left(\varphi_\lambda|_{C^1(\overline{\Omega})},\hat{u}\right)
\mbox{ and } C_k\left(\varphi_\lambda^-|_{C^1(\overline{\Omega})},\hat{v}\right)=
C_k\left(\varphi_\lambda|_{C^1(\overline{\Omega})},\hat{v}\right)
\end{equation}
for all $k\in\NN_0$.

But on account of Theorem 6.6.26 of Papageorgiou, R\u{a}dulescu \& Repov\v{s} \cite[p. 545]{13Pap-Rad-Rep}, we have
\begin{equation}\label{eq46}
  C_k(\varphi_\lambda^+,\hat{u})= C_k(\varphi_\lambda,\hat{u}) \mbox{ and } C_k(\varphi_\lambda^-,\hat{v})= C_k(\varphi_\lambda,\hat{v})
\end{equation}
for all $k\in\NN_0$.

Since $\varphi_\lambda\in C^2(H^1(\Omega))$, we can infer from \eqref{eq43}, \eqref{eq44}, \eqref{eq46} and Proposition 6.5.9 of Papageorgiou, R\u{a}dulescu \& Repov\v{s} \cite[p. 529]{13Pap-Rad-Rep} that
\begin{equation}\label{eq47}
   C_k(\varphi_\lambda,\hat{u})= C_k(\varphi_\lambda,\hat{v})=\delta_{k,1}\ZZ \mbox{ for all }k\in\NN_0.
\end{equation}

Recall that
\begin{equation}\label{eq48}
  C_k(\varphi_\lambda,0)=\delta_{k,d_m}\ZZ \mbox{ (see \eqref{eq36}). }
\end{equation}

Moreover, from Corollary 6.2.40 of Papageorgiou, R\u{a}dulescu \& Repov\v{s} \cite[p. 449]{13Pap-Rad-Rep}, we have
\begin{equation}\label{eq49}
  C_k(\varphi_\lambda,y_0)=C_k(\varphi_\lambda,\hat{y})=0 \mbox{ for }k\not\in[d_m,d_{m+1}] \mbox{ (recall that $d_m\geq2$).}
\end{equation}

We can infer from  \eqref{eq41}, \eqref{eq43}, \eqref{eq47}, \eqref{eq48}, \eqref{eq49} that
\begin{eqnarray*}
   && \tilde{y}\not\in\{u_0,v_0,\hat{u},\hat{v},0,y_0,\hat{y}\}, \\
   &\Rightarrow& \tilde{y}\in C^1(\overline{\Omega}) \mbox{ is the seventh nontrivial solution of  }(P_\lambda) \\
   && (\lambda\in(\hat{\lambda}_{m+1}-\hat{\delta},\hat{\lambda}_{m+1})).
\end{eqnarray*}
The proof is now complete.
\end{proof}

So, summarizing our findings for problem $(P_\lambda)$, we can state the following multiplicity theorem.

\begin{theorem}
  If hypotheses $H_0$, $H_1$ hold, then there exists $\hat{\delta}>0$ such that for all $\lambda\in(\hat{\lambda}_{m+1}-\hat{\delta},\hat{\lambda}_{m+1})$ problem $(P_\lambda)$ has at least seven distinct nontrivial smooth solutions
\begin{eqnarray*}
  && u_0,\hat{u}\in{\rm int}\,C_+,\ v_0,\hat{v}\in-{\rm int}\,C_+,\ y_0,\hat{y}\in C^1(\overline{\Omega}) \mbox{ nodal } \\
  && \tilde{y}\in C^1(\overline{\Omega}).
\end{eqnarray*}
\end{theorem}

{\em Open problem.} Is it possible to show that $\tilde{y}$ is nodal (see \cite{3Hu-Pap, 8Pap-Pap})? Also, it seems that we cannot generate more than seven solutions without symmetry hypotheses (see \cite{1Cas-Cas-Vel}).

\subsection*{Acknowledgments} This research was supported by the Slovenian Research Agency grants
P1-0292, J1-8131, N1-0114, N1-0064, and N1-0083. 
 The authors  thank the referees for comments and suggestions.


\begin{thebibliography}{99}
{\small
\bibitem{1Cas-Cas-Vel} A. Castro, J. Cassio, C. Velez,
Existence of seven solutions for an asymptotically linear Dirichlet problem without symmetries, {\it Annali Mat. Pura Appl.} {\bf 192} (2013), 607-619.

\bibitem{2DAg-Mar-Pap} G. D'Agui, S. Marano, N.S. Papageorgiou,
Multiple solutions to a Robin problem with indefinite weight and asymmetric reaction, {\it J. Math. Anal. Appl.} {\bf 433} (2016), 1821-1845.

\bibitem{3Hu-Pap} S. Hu, N.S. Papageorgiou,
Semilinear Robin problems with indefinite potential and competition phenomena, {\it Acta Appl. Math.}, in press (DOI: 10.1007/s10440-019-00284-y).

\bibitem{4Mar-Sac} A. Marino, C. Saccon,
Some variational theorems of mixed type and elliptic problems with jumping nonlinearities, {\it Ann. Scuola Norm. Sup. Pisa, Cl. Sci.} {\bf 25} (1997), 631-665.

\bibitem{5Mug} D. Mugnai,
Multiplicity of critical points in presence of linking: application to a superlinear boundary value problem, {\it Nonlin. Differ. Equ. Appl. (NoDEA)} {\bf 11} (2004), 379-391.

\bibitem{6Mug} D. Mugnai,
Addendum to: Multiplicity of critical points in presence of linking: application to a superlinear boundary value problem, {\it Nonlin. Differ. Equ. Appl. (NoDEA)} {\bf 11} (2004), no. 3, 379-391 and a  comment on the generalized Ambroseti-Rabinowitz condition, {\it Nonlin. Differ. Equ. Appl. (NoDEA)} {\bf 19} (2012), 299-311.

\bibitem{7Ou-Li} Z.Q. Ou, C. Li,
Existence of three nontrivial solutions for a class of superlinear elliptic equations, {\it J. Math. Anal. Appl.} {\bf 390} (2012), 418-426.

\bibitem{8Pap-Pap} N.S. Papageorgiou, F. Papalini,
Seven solutions with sign information for sublinear equations with unbounded and indefinite potential and no symmetries, {\it Israel J. Math.} {\bf 201} (2014), 761-796.

\bibitem{9Pap-Rad} N.S. Papageorgiou, V.D. R\u{a}dulescu,
Multiple solutions with precise sign for nonlinear parametric Robin problems, {\it J. Differential Equations} {\bf 256} (2014), 2449-2479.

\bibitem{prans} N.S. Papageorgiou, V.D. R\u{a}dulescu, Nonlinear nonhomogeneous Robin problems with superlinear reaction term, {\it Adv. Nonlinear Stud.} {\bf 16} (2016), no. 4, 737-764.

\bibitem{10Pap-Rad} N.S. Papageorgiou, V.D. R\u{a}dulescu,
Robin problems with indefinite unbounded potential and reaction of arbitrary growth, {\it Rev. Mat. Complut.} {\bf 19} (2016), 91-126.

\bibitem{11Pap-Rad} N.S. Papageorgiou, V.D. R\u{a}dulescu,
Robin problems near resonance at any nonprincipal eigenvalue, {\it Results Math.} {\bf 71} (2017), 1389-1412.

\bibitem{12Pap-Rad-Rep} N.S. Papageorgiou, V.D. R\u{a}dulescu, D.D. Repov\v{s},
Positive solutions for perturbations of the Robin eigenvalue problem plus an indefinite potential, {\it Discr. Cont. Dyn. Syst.} {\bf 37} (2017), 2589-2618.

\bibitem{13Pap-Rad-Rep} N.S. Papageorgiou, V.D. R\u{a}dulescu, D.D. Repov\v{s},
{\it Nonlinear Analysis - Theory and Methods}, Springer Monographs in Mathematics, Springer, Cham, 2019.

\bibitem{14Pap-Win} N.S Papageorgiou, P. Winkert,
{\it Applied Nonlinear Functional Analysis}, Walter De Gruyter, Berlin, 2018.

\bibitem{pawi} N.S. Papageorgiou, P. Winkert, Double resonance for Robin problems with indefinite and unbounded potential, {\it Discrete Contin. Dyn. Syst. Ser. S} {\bf 11} (2018), no. 2, 323-344.

\bibitem{pzana} N.S. Papageorgiou, C. Zhang, Noncoercive resonant $(p,2)$-equations with concave terms, {\it Adv. Nonlinear Anal.} {\bf 9} (2020), no. 1, 228-249.

\bibitem{15Rab-Su-Wan} P. Rabinowitz, J. Su, Z.Q. Wang,
Multiple solutions of superlinear elliptic equations, {\it Rend. Lincei Mat. Appl.} {\bf 18} (2007), 97-108.

\bibitem{rolando} S. Rolando, Multiple nonradial solutions for a nonlinear elliptic problem with singular and decaying radial potential, {\it Adv. Nonlinear Anal.} {\bf 8} (2019), no. 1, 885-901.

\bibitem{16Su-Zen} J. Su, R. Zeng,
Multiple periodic solutions of superlinear ordinary differential equations with a parameter, {\it Nonlinear Anal.} {\bf 74} (2011), 6442-6450.

\bibitem{17Wan} S. Wang,
Neumann problems of semilinear elliptic equations involving critical Sobolev exponents, {\it J. Differential Equations} {\bf 93} (1991), 283-310.
}
\end{thebibliography}
\end{document}